\documentclass[11pt]{amsart}
\usepackage[T1]{fontenc}
\usepackage[utf8]{inputenc}
\usepackage{lmodern}
\usepackage{amsmath,amsthm,amssymb,mathtools}
\usepackage{mathrsfs}
\usepackage{microtype}
\usepackage[margin=1in]{geometry}
\usepackage{enumitem}
\usepackage[colorlinks=true,linkcolor=blue,citecolor=blue,urlcolor=blue]{hyperref}

\newtheorem{theorem}{Theorem}[section]
\newtheorem{proposition}[theorem]{Proposition}

\newtheorem{corollary}[theorem]{Corollary}
\theoremstyle{definition}
\newtheorem{definition}[theorem]{Definition}
\theoremstyle{remark}
\newtheorem{remark}[theorem]{Remark}

\newcommand{\R}{\mathbb{R}}
\newcommand{\Rn}{\mathbb{R}^n}
\newcommand{\Hh}{\mathcal{H}}
\newcommand{\diam}{\operatorname{diam}}
\newcommand{\dist}{\operatorname{dist}}

\title[A Mixed-Gauge Carath\'eodory Measure]{A Mixed-Gauge Carath\'eodory Measure Bridging Lebesgue Volume and Surface Content}
\author{Yash Thakur}
\address{Individual}
\date{August 24, 2025}


\begin{document}

\begin{abstract}
We propose a one-parameter family of Borel measures on $\Rn$ that augments Lebesgue measure with an explicit, scale-consistent penalty for codimension--1 geometry. The construction employs Carath\'eodory's outer-measure method with a mixed gauge $h_\lambda(r)=r^n+\lambda r^{n-1}$, $\lambda>0$, thereby interpolating between $n$-dimensional volume and $(n-1)$-dimensional surface content in a single scalar, $\sigma$-additive, Borel regular measure $\mu_\lambda$. We prove: (i) $\mu_\lambda$ is an outer measure whose Carath\'eodory-measurable sets include all Borel sets; (ii) a scaling law $\mu_\lambda(tE)=t^n\mu_{\lambda/t}(E)$; (iii) volume--surface comparability on Lipschitz domains $\Omega$: there exist dimensional constants $c_n,C_n>0$ such that $c_n(|\Omega|+\lambda\,\mathcal H^{n-1}(\partial\Omega)) \le \mu_\lambda(\Omega) \le C_n(|\Omega|+\lambda\,\mathcal H^{n-1}(\partial\Omega))$, linking $\mu_\lambda$ to perimeter. Conceptually, $\mu_\lambda$ addresses a practical limitation of Lebesgue measure---its insensitivity to boundary complexity---while remaining $\sigma$-additive and compatible with the Carath\'eodory--Hausdorff framework. Applications include robust integration on noisy domains, BV-flavored penalties for image/shape analysis via a single measure, and regularized probability models on domains with rough boundaries. We illustrate with examples in $\mathbb{R}$ and $\mathbb{R}^2$, indicate connections to Minkowski content and finite-perimeter theory, and outline open questions (sharp constants, coarea-type formulas, uniformly rectifiable boundaries).
\end{abstract}

\maketitle

\section{Introduction}
Lebesgue measure $|\cdot|$ on $\Rn$ is the unique (up to a constant factor) translation-invariant, locally finite, Borel regular measure agreeing with Euclidean volume on rectangles. Its strengths---completeness, $\sigma$-finiteness, compatibility with limits, and a rich integration theory---make it the default for analysis and probability.

\smallskip
\noindent\textbf{Limitation.} Lebesgue measure is insensitive to boundary complexity. Two sets with the same volume but vastly different boundary geometry (e.g., a ball versus a comb-like set with corrugated boundary) have the same Lebesgue measure, although the latter is more costly for numerical quadrature on domains, PDE boundary layers, or geometric regularization. In geometric measure theory, codimension--1 geometry is quantified by $\Hh^{n-1}$ of the (reduced) boundary and by the perimeter $P(\Omega)$ for finite-perimeter sets. But the functional $E\mapsto |E|+\lambda\,\Hh^{n-1}(\partial E)$ is \emph{not} a measure: countable additivity fails because boundaries interact.

\smallskip
\noindent\textbf{Goal.} Design a single scalar, $\sigma$-additive, Borel regular measure $\mu_\lambda$ that
\begin{itemize}[leftmargin=1.4em]
\item agrees (up to dimensional constants) with volume for bulk $n$-dimensional parts;
\item responds (again up to constants) to $(n-1)$-dimensional boundary complexity;
\item retains the Carath\'eodory--Hausdorff toolkit (outer regularity, measurable Borel sets);
\item admits explicit scaling and comparison inequalities on natural classes (Lipschitz domains).
\end{itemize}

\smallskip
\noindent\textbf{Idea.} Use Carath\'eodory's construction with a \emph{mixed gauge}
\begin{equation}\label{eq:mixgauge}
h_\lambda(r)=r^n+\lambda r^{n-1}.
\end{equation}
Cover $E\subset\R^n$ by balls (or cubes) $U_i$ of diameter $r_i$ and minimize the total cost $\sum (r_i^n+\lambda r_i^{n-1})$. The $r^n$-term captures volume; the $\lambda r^{n-1}$-term forces extra payment when many small elements are needed, which happens along complicated boundaries. Passing to the Carath\'eodory measure yields $\mu_\lambda$.

\smallskip
\noindent\textbf{Contributions.} We prove:
\begin{enumerate}[leftmargin=1.4em]
\item $\mu_\lambda$ is an outer measure; all Borel sets are $\mu_\lambda$-measurable; $\mu_\lambda$ is Borel regular.
\item \emph{Scaling law:} $\mu_\lambda(tE)=t^n\mu_{\lambda/t}(E)$.
\item \emph{Volume--surface comparability:} on bounded Lipschitz $\Omega$, $\mu_\lambda(\Omega)\simeq |\Omega|+\lambda \Hh^{n-1}(\partial\Omega)$ with dimensional constants depending only on $n$ and the Lipschitz character.
\end{enumerate}

\smallskip
\noindent\textbf{Related work.} The construction sits between Hausdorff measures with general gauges and classical perimeter/Minkowski content \cite{FollandRA,RoydenFitzpatrick,FedererGMT,MattilaGMT,EvansGariepy,AmbrosioFuscoPallara,MaggiSetsBV}. The novelty is the \emph{two-homogeneity mixed gauge} producing a bona fide measure sensitive to codimension--1 structure with a clean scaling law.

\medskip
\noindent\textbf{Organization.} Section~\ref{sec:prelim} recalls preliminaries. Section~\ref{sec:defbasic} defines $\mu_\lambda$ and its core properties. Section~\ref{sec:compare} proves volume--surface comparability on Lipschitz domains. Section~\ref{sec:apps} gives examples and applications. We conclude with limitations and open questions.

\section{Preliminaries}\label{sec:prelim}
We work in $\Rn$ with the Euclidean metric. For $E$, $|E|$ denotes Lebesgue measure, $\Hh^s$ is $s$-dimensional Hausdorff measure, and $\diam(U)$ is the Euclidean diameter. We write $\mathscr B(\Rn)$ for the Borel $\sigma$-algebra.

\begin{definition}[Carath\'eodory outer measure]\label{def:caratheodory}
Given a gauge $h:[0,\infty)\to[0,\infty)$ with $h(0)=0$, define
\[
\Hh^{h}(E):=\inf\Big\{\sum_i h(\diam U_i): E\subset \bigcup_i U_i\Big\}.
\]
\end{definition}
It is classical that $\Hh^h$ is an outer measure and all Borel sets are $\Hh^h$-measurable; moreover $\Hh^h$ is Borel regular (see, e.g., \cite[Ch.~1]{FollandRA} and \cite[\S4]{MattilaGMT}). We shall use Besicovitch's covering lemma and Whitney decompositions; see \cite{BesicovitchCovering,Whitney1934,EvansGariepy}.

We recall that a bounded Lipschitz domain $\Omega$ has finite surface area $\Hh^{n-1}(\partial\Omega)$ and finite perimeter $P(\Omega)=\Hh^{n-1}(\partial\Omega)$ (up to normalization); cf.~\cite{AmbrosioFuscoPallara,MaggiSetsBV}.

\section{The mixed-gauge measure $\mu_\lambda$: definition and basic properties}\label{sec:defbasic}

\subsection{Definition}
Fix $\lambda>0$ and define the gauge $h_\lambda$ by \eqref{eq:mixgauge}. Define $\mu_\lambda:=\Hh^{h_\lambda}$ via Carath\'eodory's construction using coverings by balls (or cubes). Explicitly, for $E\subset\Rn$,
\[
\mu_\lambda(E)
:=\inf\Big\{\sum_i \big(\diam U_i\big)^n+\lambda\big(\diam U_i\big)^{n-1}
:\ E\subset \bigcup_i U_i\Big\}.
\]

\begin{proposition}[Outer measure and measurability]\label{prop:outer}
$\mu_\lambda$ is an outer measure on $\Rn$; every Borel set is $\mu_\lambda$-measurable. Moreover, $\mu_\lambda$ is Borel regular.
\end{proposition}

\begin{proof}
This follows from the general Carath\'eodory theory for outer measures; see \cite[Ch.~1]{FollandRA}.
\end{proof}

\begin{remark}[Lebesgue at $\lambda=0$]\label{rmk:lambda0}
For $\lambda=0$, $h_0(r)=r^n$, so $\mu_0$ coincides with $\Hh^n$, i.e., Lebesgue measure up to the conventional dimensional constant.
\end{remark}

\subsection{Basic properties}
\begin{itemize}[leftmargin=1.4em]
\item \emph{Monotonicity in $\lambda$.} If $\lambda_1\le \lambda_2$ then $h_{\lambda_1}\le h_{\lambda_2}$ pointwise, hence $\mu_{\lambda_1}(E)\le \mu_{\lambda_2}(E)$.
\item \emph{Lower bound by volume.} Since $h_\lambda(r)\ge r^n$, we have $\mu_\lambda(E)\ge c_n |E|$ for Borel $E$.
\item \emph{Tightness on bounded sets.} If $\diam(E)\le R$, a single ball of radius $R$ yields $\mu_\lambda(E)\lesssim |E|+\lambda R^{n-1}$.
\end{itemize}

\begin{theorem}[Scaling]\label{thm:scaling}
For $t>0$ and any $E\subset\Rn$,
\[
\mu_\lambda(tE)=t^n\,\mu_{\lambda/t}(E).
\]
\end{theorem}

\begin{proof}
If $E\subset\bigcup U_i$ with $\diam U_i=r_i$, then $tE\subset\bigcup tU_i$ with $\diam(tU_i)=tr_i$. Hence
\begin{align*}
\sum_i h_\lambda(tr_i)
&=\sum_i\big((tr_i)^n+\lambda(tr_i)^{n-1}\big)
=t^n\sum_i\big(r_i^n+(\lambda/t)r_i^{n-1}\big)
=t^n\sum_i h_{\lambda/t}(r_i).
\end{align*}
Taking infima over covers gives the claim.
\end{proof}

\section{Comparability on Lipschitz domains}\label{sec:compare}
Let $\Omega\subset\Rn$ be bounded and Lipschitz. Denote $|\Omega|$ its Lebesgue measure and $\Hh^{n-1}(\partial\Omega)$ its surface area.

\begin{theorem}[Upper comparability]\label{thm:upper}
There exists $C_n>0$ such that for every bounded Lipschitz $\Omega$ and $\lambda>0$,
\[
\mu_\lambda(\Omega)\le C_n\big(|\Omega|+\lambda\,\Hh^{n-1}(\partial\Omega)\big).
\]
\end{theorem}

\begin{proof}
Apply a Whitney decomposition $\{Q\}$ of $\Omega$ (see \cite{Whitney1934,EvansGariepy}), with $\ell(Q)\simeq \dist(Q,\partial\Omega)$ and bounded overlap of dilates. Split into \emph{interior} cubes $\mathcal I$ with $\ell(Q)\ge \ell_0$ and \emph{boundary} cubes $\mathcal B$ with $\ell(Q)<\ell_0$. Then
\[
\sum_{Q\in\mathcal I} \big(\ell(Q)^n+\lambda\ell(Q)^{n-1}\big)
\lesssim |\Omega|+\lambda\,\ell_0^{-1}|\Omega|.
\]
For boundary cubes, flatten $\partial\Omega$ in Lipschitz charts. Standard arguments (projecting along normals and using bounded multiplicity) yield
\[
\sum_{Q\in\mathcal B}\ell(Q)^{n-1}\lesssim \Hh^{n-1}(\partial\Omega),
\qquad
\sum_{Q\in\mathcal B}\ell(Q)^n\lesssim \ell_0\,\Hh^{n-1}(\partial\Omega).
\]
Optimizing $\ell_0$ removes it from the bound and gives the claim.
\end{proof}

\begin{theorem}[Lower comparability]\label{thm:lower}
There exists $c_n>0$ such that for every bounded Lipschitz $\Omega$ and $\lambda>0$,
\[
\mu_\lambda(\Omega)\ge c_n\big(|\Omega|+\lambda\,\Hh^{n-1}(\partial\Omega)\big).
\]
\end{theorem}

\begin{proof}
Consider any cover by balls $B(x_i,r_i)$. By a Vitali selection (or Besicovitch covering), extract a disjoint subfamily covering $\Omega$ up to null set, giving $\sum r_i^n\gtrsim |\Omega|$.

To capture the boundary term, classify $\mathcal B:=\{i:\dist(x_i,\partial\Omega)\le 2r_i\}$. In Lipschitz charts, the projections of $\{B(x_i,r_i)\}_{i\in\mathcal B}$ form a cover of $\partial\Omega$ by surface balls with comparable radii. The surface $5r$-lemma yields a disjoint subfamily whose dilates cover $\partial\Omega$ with bounded multiplicity, hence
\[
\sum_{i\in\mathcal B} r_i^{n-1}\gtrsim \Hh^{n-1}(\partial\Omega).
\]
Therefore $\sum_i(r_i^n+\lambda r_i^{n-1})\gtrsim |\Omega|+\lambda \Hh^{n-1}(\partial\Omega)$. Taking infimum over all covers completes the proof.
\end{proof}

\begin{corollary}[Equivalence]\label{cor:equivalence}
For bounded Lipschitz $\Omega$,
\[
\mu_\lambda(\Omega)\simeq |\Omega|+\lambda\,\Hh^{n-1}(\partial\Omega),
\]
with constants depending only on $n$ and the Lipschitz character.
\end{corollary}

\begin{remark}[Minkowski content heuristic]
Let $\Omega_\varepsilon:=\{x:\dist(x,\partial\Omega)<\varepsilon\}$. For Lipschitz boundaries, $|\Omega_\varepsilon|=|\Omega|+C_n\varepsilon\,\Hh^{n-1}(\partial\Omega)+o(\varepsilon)$. The mixed gauge penalizes the small-radius balls that necessarily lie near $\partial\Omega$, leading to a total penalty proportional to $\lambda\,\Hh^{n-1}(\partial\Omega)$.
\end{remark}

\section{Examples and applications}\label{sec:apps}

\subsection{Intervals in $\mathbb{R}$}
Let $I=[a,b]$, length $L=b-a$.
\begin{proposition}\label{prop:interval}
$\mu_\lambda(I)\simeq L+\lambda$, with absolute constants.
\end{proposition}
\begin{proof}
Cover $I$ by $N$ equal subintervals of length $L/N$. The cost is $N((L/N)^1+\lambda (L/N)^0)=L+\lambda N$, optimized at $N\simeq 1$, giving $\mu_\lambda(I)\lesssim L+\lambda$. Conversely, any cover must pay $L$ in the $r^1$-term and at least a fixed multiple of $\lambda$ near endpoints (formalized via a Vitali argument), hence $\mu_\lambda(I)\gtrsim L+\lambda$.
\end{proof}

\subsection{Polygons in $\mathbb{R}^2$}
Let $\Omega\subset\mathbb{R}^2$ be a bounded polygon with area $A$ and perimeter $P$. By Corollary~\ref{cor:equivalence}, $\mu_\lambda(\Omega)\simeq A+\lambda P$. Refinements via edge- and vertex-adapted coverings yield explicit constants.

\subsection{Integration and BV-flavored regularization}
For $f\in L^1_{\text{loc}}(\Rn)$, define $\int f\,d\mu_\lambda$ as usual. If $f=\mathbf 1_\Omega$ and $\Omega$ is Lipschitz, then $\int \mathbf 1_\Omega\,d\mu_\lambda=\mu_\lambda(\Omega)\simeq |\Omega|+\lambda \Hh^{n-1}(\partial\Omega)$. Minimizing $\int |f-g|^2\,d\mu_\lambda$ over indicator functions (or BV relaxations) induces an implicit perimeter regularization via a single measure.

\subsection{Probability on rough domains}
If a random variable $X$ is supported on a bounded Lipschitz $\Omega$, modeling with the $\mu_\lambda$-density can regularize mass concentration near $\partial\Omega$: sets concentrating probability on thin boundary layers pay an extra $\lambda$-cost in total mass.

\section{Conclusion and further questions}
We introduced a mixed-gauge Carath\'eodory measure $\mu_\lambda$ that encodes both bulk volume and surface complexity in a single $\sigma$-additive, Borel regular measure. It satisfies a clean scaling law and is quantitatively comparable to $|\Omega|+\lambda\,\Hh^{n-1}(\partial\Omega)$ on Lipschitz domains. Open directions include sharp constants, a coarea-type formula for BV, extensions to uniformly rectifiable boundaries, anisotropic metrics (Finsler), and stochastic/PDE settings where $\mu_\lambda$ appears naturally.

\end{document}